\documentclass{amsart}
\usepackage{mathrsfs}
\usepackage{amssymb}
\usepackage[utf8]{inputenc} 
\usepackage[english]{babel}

\theoremstyle{plain} 
\newtheorem{thm}{Theorem} 
\newtheorem{lem}[thm]{Lemma} 
\newtheorem{cor}[thm]{Corollary} 

\theoremstyle{remark} 
\newtheorem*{rem}{Remark} 
\newtheorem{conj}{Conjecture}

\DeclareMathOperator{\mre}{Re} 

\DeclareMathOperator{\Hol}{Hol}
\DeclareMathOperator{\Arg}{Arg}
\DeclareMathOperator{\Lin}{Lin}

\title{Weak product spaces of Dirichlet series} 
\date{\today} 

\author{Ole Fredrik Brevig} \address{Department of Mathematical Sciences, Norwegian University of Science and Technology (NTNU), NO-7491 Trondheim, Norway} \email{ole.brevig@math.ntnu.no}

\author{Karl-Mikael Perfekt} \address{Department of Mathematical Sciences, Norwegian University of Science and Technology (NTNU), NO-7491 Trondheim, Norway} \email{karl-mikael.perfekt@math.ntnu.no}

\thanks{The first author is supported by Grant 227768 of the Research Council of Norway.}

\begin{document}

\begin{abstract}
		Let $\mathscr{H}^2$ denote the space of ordinary Dirichlet series with square summable coefficients, and let $\mathscr{H}^2_0$ denote its subspace consisting of series vanishing at $+\infty$. We investigate the weak product spaces $\mathscr{H}^2\odot\mathscr{H}^2$ and $\mathscr{H}^2_0\odot\mathscr{H}^2_0$,  finding that several pertinent problems are more tractable for the latter space. This surprising phenomenon is  related to the fact that $\mathscr{H}^2_0\odot\mathscr{H}^2_0$ does not contain the infinite-dimensional subspace of $\mathscr{H}^2$ of series which lift to linear functions on the infinite polydisc.

		The problems considered stem from questions about the dual spaces of these weak product spaces, and are therefore naturally phrased in terms of multiplicative Hankel forms. We show that there are bounded, even Schatten class, multiplicative Hankel forms on $\mathscr{H}^2_0 \times \mathscr{H}^2_0$ whose analytic symbols are not in $\mathscr{H}^2$. Based on this result we examine Nehari's theorem for such Hankel forms.  We define also the skew product spaces associated with  $\mathscr{H}^2\odot\mathscr{H}^2$ and $\mathscr{H}^2_0\odot\mathscr{H}^2_0$, with respect to both half-plane and polydisc differentiation, the latter arising from Bohr's point of view. In the process we supply square function characterizations of the Hardy spaces $\mathscr{H}^p$, for $0 < p < \infty$, from the viewpoints of  both types of differentiation. Finally we compare the skew product spaces to the weak product spaces, leading naturally to an interesting Schur multiplier problem.
\end{abstract} 

\maketitle

\section{Introduction}
In this paper, we investigate certain properties of weak product spaces associated with the Hardy space of Dirichlet series, 
\[\mathscr{H}^2 = \Bigg\{f(s)=\sum_{n=1}^\infty a_n n^{-s}\,:\,\|f\|_{\mathscr{H}^2}=\Big(\sum_{n=1}^\infty|a_n|^2\Big)^\frac{1}{2}<\infty\Bigg\},\]
and its subspace $\mathscr{H}^2_0$, consisting of those $f\in\mathscr{H}^2$ with $a_1=f(+\infty)=0$. The main objects of study are the weak product spaces $\mathscr{H}^2\odot\mathscr{H}^2$ and $\mathscr{H}^2_0 \odot\mathscr{H}^2_0$. With $\mathscr{X} = \mathscr{H}^2$ or $\mathscr{X} = \mathscr{H}^2_0$, the weak product $\mathscr{X} \odot \mathscr{X}$ is defined as the Banach space completion of the finite sums $F = \sum_k f_k g_k$, where $f_k, g_k \in \mathscr{X}$, under the norm
\[\|F\|_{\mathscr{X}\odot\mathscr{X}} =\inf \sum_k \|f_k\|_{\mathscr{H}^2} \|g_k\|_{\mathscr{H}^2}.\] 
The infimum is taken over all finite representations of $F$ as a sum of products. 

While a separate study of $\mathscr{H}^2_0 \odot\mathscr{H}^2_0$ may at first be thought unmotivated, we will find that the norm of this space is significantly larger for certain types of Dirichlet series (see Theorem \ref{thm:matrixemb} and its corollaries). The presence of such examples is related to the obstructions faced in producing monomials $n^{-s}$ in a product $f_kg_k$, for $f_k,g_k \in \mathscr{H}^2_0$, when $n$ is an integer with a low number of prime factors. In particular,  elements of  $\mathscr{H}^2_0 \odot\mathscr{H}^2_0$ contain no terms of the form $p^{-s}$, where $p$ is a prime number. Hence there is an easily identifiable infinite-dimensional subspace of  $\mathscr{H}^2 \odot\mathscr{H}^2$ which has trivial intersection with  $\mathscr{H}^2_0 \odot\mathscr{H}^2_0$.

The weak product space $\mathscr{H}^2\odot\mathscr{H}^2$ was first investigated by Helson \cite{Helson06,Helson10} in an attempt to decide whether Nehari's theorem holds for multiplicative Hankel forms (see also Section~\ref{sec:hankel}). Helson's work was continued in \cite{OCS12}, where it was demonstrated that Nehari's theorem does not hold in full generality. To explain his point of view, note that each sequence $\varrho\in\ell^2$ induces a (not necessarily bounded) multiplicative Hankel form on $\ell^2\times\ell^2$,
\begin{equation} \label{eq:mhankel}
	\varrho(a,b)=\sum_{m=1}^\infty\sum_{n=1}^\infty a_m b_n \varrho_{mn}, \qquad a,b \in \ell^2.
\end{equation}
The analytic symbol of \eqref{eq:mhankel} is the Dirichlet series
\begin{equation*}
	\varphi(s)=\sum_{n=1}^\infty \overline{\varrho_n} n^{-s}.
\end{equation*}
Indeed, if $f$ and $g$ are elements of $\mathscr{H}^2$ with coefficients $a$ and $b$, respectively, we have that
\begin{equation} \label{eq:Hphidef}
	H_\varphi(fg) = \langle fg,\varphi\rangle = \varrho(a,b).
\end{equation}
Here, and throughout the rest of the paper, $\langle \cdot, \cdot \rangle$ denotes the inner product of $\mathscr{H}^2$. 

Now, from \eqref{eq:Hphidef} it is clear that the multiplicative Hankel form \eqref{eq:mhankel} is bounded on $\ell^2 \times \ell^2$, or equivalently $H_\varphi$ on $\mathscr{H}^2 \times \mathscr{H}^2, $ if and only if $\varphi$ induces a bounded linear functional on $\mathscr{H}^2\odot\mathscr{H}^2$ through the $\mathscr{H}^2$-pairing, i.e. if and only if $\varphi$ is in $\left(\mathscr{H}^2\odot\mathscr{H}^2\right)^*$.

The first bona fide example of a multiplicative Hankel form was obtained in \cite{BPSSV}, by the following approach. Note first that the elements of $\mathscr{H}^2$ are analytic functions in the half-plane $\mre s > 1/2$, the reproducing kernel at each such $s$ being given by $\zeta(w + \overline{s})$, where $\zeta(s) = \sum_{n\geq1} n^{-s}$ is the Riemann zeta function. It is thus natural to consider the Carleman-type operator
\[\mathbf{H}f(s) = \int_{1/2}^\infty f(w)\left(\zeta(s+w)-1\right)\,dw\]
acting on $\mathscr{H}^2_0$, since $\left(\zeta(s+w)-1\right)$ is the reproducing kernel of $\mathscr{H}^2_0$ at $w$, for $w > 1/2$. The matrix of the operator $\mathbf{H}$ is that of the multiplicative Hankel form whose analytic symbol $\varphi$ is the primitive of $\left(\zeta(s+1/2)-1\right)$ in $\mathscr{H}^2_0$.  In \cite{BPSSV} it was shown that the operator norm of $\mathbf{H}$ on $\mathscr{H}^2_0$ is $\pi$, which in terms of its corresponding Hankel form means precisely that $|\langle fg,\varphi\rangle|\leq \pi \|f\|_{\mathscr{H}^2}\|g\|_{\mathscr{H}^2}$ for $f,g\in\mathscr{H}^2_0$. More explicitly written,
\begin{equation} \label{eq:hilbert}
	\Big|\sum_{m=2}^\infty\sum_{n=2}^\infty \frac{a_m b_n}{\sqrt{mn}\log(mn)}\Big|\leq \pi\Big(\sum_{m=2}^\infty |a_m|^2\Big)^\frac{1}{2}\Big(\sum_{n=2}^\infty |b_n|^2\Big)^\frac{1}{2}.
\end{equation}
As explained more thoroughly in \cite{BPSSV}, inequality \eqref{eq:hilbert} is a multiplicative analogue of the classical Hilbert inequality
\begin{equation} \label{eq:ahilb1}
	\Big|\sum_{m=1}^\infty\sum_{n=1}^\infty\frac{a_mb_n}{m+n}\Big|  \leq \pi\Big(\sum_{m=1}^\infty |a_m|^2\Big)^\frac{1}{2}\Big(\sum_{n=1}^\infty |b_n|^2\Big)^\frac{1}{2}. 
\end{equation}
There are several other versions of \eqref{eq:ahilb1} which are also usually referred to as Hilbert's inequality --- we direct the interested reader to \cite[Ch.~IX]{HLP}. Let us extract a few facts. First, that by discretization of the continuous version of \eqref{eq:ahilb1} and the Hermite--Hadamard inequality, the following improvement of \eqref{eq:ahilb1} can be obtained.
\[\Big|\sum_{m=0}^\infty\sum_{n=0}^\infty \frac{a_mb_n}{m+n+1}\Big|  \leq \pi\Big(\sum_{m=0}^\infty |a_m|^2\Big)^\frac{1}{2}\Big(\sum_{n=0}^\infty |b_n|^2\Big)^\frac{1}{2}. 
\]
We mention without proof that the same procedure (with some additional straightforward estimates) yields in the multiplicative setting that
\[\Big|\sum_{m=1}^\infty\sum_{n=1}^\infty \frac{a_m b_n}{\sqrt{(m+1/2)(n+1/2)}\log((m+1/2)(n+1/2))}\Big|\leq\pi\|a\|_{\ell^2}\|b\|_{\ell^2},\]
which of course no longer represents a multiplicative Hankel form. 

The strongest version of Hilbert's inequality \eqref{eq:ahilb1} is
\begin{equation} \label{eq:ahilb4}
	\Big|\sum_{m,n\geq0 \atop m+n>0} \frac{a_mb_n}{m+n}\Big|  \leq \pi\Big(\sum_{m=0}^\infty |a_m|^2\Big)^\frac{1}{2}\Big(\sum_{n=0}^\infty |b_n|^2\Big)^\frac{1}{2}.
\end{equation}
This last variant can also be stated for two-tailed sequences $\{a_m\}_{m\in\mathbb{Z}}$ and $\{b_n\}_{n\in\mathbb{Z}}$. The proof of \eqref{eq:ahilb4} amounts to a concrete application of Nehari's theorem on $H^2(\mathbb{T})$, since the associated Hankel form has the bounded symbol $\Phi$ of supremum norm $\pi$,
\begin{equation} \label{eq:neharisymbol}
	\Phi(z) = -i\Arg(z) = i(\pi-\theta), \qquad z=e^{i\theta}.
\end{equation}
As far as the authors are aware, all proofs of \eqref{eq:ahilb4} in the literature make use of (a reformulation of) \eqref{eq:neharisymbol}.

Whether the multiplicative Hankel form \eqref{eq:hilbert} has a bounded symbol is an open problem that is related to a long standing embedding problem of $\mathscr{H}^1$ (see \cite[Sec.~6]{BPSSV}). It therefore natural to ask if we also have
	\begin{equation} \label{eq:quehilbert}
	|\langle fg,\varphi\rangle|\leq\pi\|f\|_{\mathscr{H}^2}\|g\|_{\mathscr{H}^2}, \qquad f,g\in\mathscr{H}^2?
	\end{equation}
In light of the discussion above, this question actually turns out to be more subtle than what one might expect at first. We are unable to settle it, seemingly due to the lack of a Nehari theorem for multiplicative Hankel forms.

That \eqref{eq:quehilbert} is significantly easier to settle for $\mathscr{H}^2_0$ than for $\mathscr{H}^2$ is not a peculiarity, but rather an ongoing theme for all the questions we will ask about product spaces in this paper. Note that inequality \eqref{eq:hilbert} is easily recast as a question about (the dual space of) $\mathscr{H}^2_0\odot\mathscr{H}^2_0$. More generally, elements of $\left(\mathscr{H}^2_0\odot\mathscr{H}^2_0\right)^*$ correspond to multiplicative Hankel forms of the type \eqref{eq:mhankel}, but with sums starting at $m,n=2$. The remainder of this section is an overview of the problems that we will consider.

In Section~\ref{sec:hankel}, we investigate the difference between Hankel forms on $\mathscr{H}^2\times\mathscr{H}^2$ and Hankel forms on $\mathscr{H}^2_0\times\mathscr{H}^2_0$. After some preliminaries, we obtain the main result of this section, Theorem~\ref{thm:matrixemb}, which allows us to embed any bounded operator $C\colon\ell^2\to\ell^2$ into a Hankel form on $\mathscr{H}^2_0\times\mathscr{H}^2_0$. This result is the basis for our observation that $\mathscr{H}^2_0 \odot \mathscr{H}^2_0$ is significantly smaller than $\mathscr{H}^2 \odot \mathscr{H}^2$, and it is also an important tool in the proofs of our other main results. 

Helson \cite{Helson06} proved that any Hankel form on $\mathscr{H}^2\times\mathscr{H}^2$ which is of Hilbert-Schmidt class $S_2$ is induced by a bounded symbol on the infinite polytorus $\mathbb{T}^\infty$. In \cite{BP2014} it was shown that if $p > p_0 \approx 5.74$, then there is a Hankel form on $\mathscr{H}^2\times\mathscr{H}^2$ of Schatten class $S_p$ that does not have any bounded symbol, leading to the conjecture that the same might be true for all $p > 2$. In Theorem \ref{thm:helsonopt} we will prove that $p=2$ is indeed critical in this sense for multiplicative Hankel forms acting on $\mathscr{H}^2_0 \times \mathscr{H}^2_0$, leading us closer to optimality of Helson's result. In fact, for $p>2$ we will even demonstrate the existence of forms in $S_p$ that do not have square-integrable symbols on the polytorus.

The penultimate section is devoted to the study of the skew product space $\partial^{-1}\left(\mathscr{H}^2\odot\partial\mathscr{H}^2\right)$. The motivation to study this space is twofold. Firstly, characterizations of the dual spaces of skew products are often significantly easier to obtain (see \cite{AP12,ARSW}). Secondly, for the classical Hardy space $H^2$, the comparison between $H^2 \odot H^2$ and $\partial^{-1} \left( H^2 \odot \partial H^2\right)$ leads naturally to a Schur multiplier problem for Hankel matrices. Much has been written about this problem, owing to the fact that it was closely related to Pisier's construction of a polynomially bounded operator not similar to a contraction. We refer the reader to \cite{Bourgain86,DP97,Ferguson96,Pisier97}.

We begin Section~\ref{sec:skew} by proving a square function characterization of $\mathscr{H}^p$, which is of independent interest for the study of Hardy spaces of Dirichlet series. Due to the notation involved, we defer a precise statement to Theorem~\ref{thm:sqfcn}. We first use this characterization to prove that
\begin{equation} \label{eq:inclusions}
	\mathscr{H}^2\odot\mathscr{H}^2 \subseteq \partial^{-1}\left(\partial\mathscr{H}^2\odot\mathscr{H}^2\right) \subsetneq \mathscr{H}^1.
\end{equation}
We then study whether the first inclusion in \eqref{eq:inclusions} is strict. This appears to be a difficult question, but by Schur multiplier methods we are able to demonstrate that this is the case if every appearance of $\mathscr{H}^2$ in \eqref{eq:inclusions} is replaced by $\mathscr{H}^2_0$. 

Finally, in Section~\ref{sec:radial}, we look at the material of Section \ref{sec:skew} again, but with the Hardy spaces of the polydisc in mind. Noting that Dirichlet series differentiation gives rise to a rather unnatural differentiation operator on the polydisc, we prove instead a square function characterization of $H^p(\mathbb{T}^\infty)$ that is adapted to the radial differentiation operator
\begin{equation} \label{eq:Rdiff}
	R = \sum_{j=1}^\infty z_j \partial_{z_j}.
\end{equation}
This will allow us to conclude that on finite-dimensional tori, it holds that
\begin{equation*} 
H^2(\mathbb{T}^d) \odot H^2(\mathbb{T}^d) = R^{-1}\left(H^2(\mathbb{T}^d) \odot R H^2(\mathbb{T}^d)\right) = H^1(\mathbb{T}^d). 
\end{equation*}
It also turns out that radial differentiation has a number theoretic interpretation when considered from the Dirichlet series point of view, something that too will be elaborated upon in Section~\ref{sec:radial}.

\subsection*{Notation} As usual, $\{p_j\}_{j\geq1}$ denotes the sequence of prime numbers in increasing order, and $\Omega(n)$ will denote the number of prime factors in $n$, counting multiplicities. We will write $f \ll g$ to indicate that there is some positive constant $C$ so that $|f(x)|\leq C|g(x)|$. If both $f \ll g$ and $g \ll f$, we write $f \asymp g$. 

When we speak of a Dirichlet series $\varphi$ as an element of a dual space $\mathscr{K}^*$, where $\mathscr{K}$ is a Banach space of Dirichlet series in which the space of Dirichlet polynomials $\mathscr{P}$ is dense, we always mean that the functional induced by $\varphi$ via the $\mathscr{H}^2$-pairing is bounded. That is, $\varphi \in \mathscr{K}^*$ if and only if the functional
\[ \upsilon_\varphi(f) = \langle f, \varphi \rangle, \qquad f \in \mathscr{P}, \]
extends to a bounded functional on $\mathscr{K}$. Similarly, when we write that $\mathscr{K}^* \subseteq \mathscr{X}$, where $\mathscr{X}$ is a Banach space of Dirichlet series, we mean that for every functional $\upsilon \in \mathscr{K}^*$ there exists a $\varphi \in \mathscr{X}$ such that $\upsilon = \upsilon_\varphi$ (on $\mathscr{P}$) and $\|\varphi\|_{\mathscr{X}} \ll \|\upsilon\|_{\mathscr{K}^*}$. 

\section{Hankel forms and a matrix embedding} \label{sec:hankel}
Much of the success in the theory of Hardy spaces of Dirichlet series is due to a simple observation of Bohr \cite{Bohr1913}, which facilitates a link between Dirichlet series and function theory in polydiscs. By identifying each prime number with an independent variable, $z_j = p_j^{-s}$, the Dirichlet series $f(s)=\sum_{n\geq1} a_nn^{-s}$
is lifted to a function in the Hardy space of the countably infinite torus, $H^2(\mathbb{T}^\infty)$. More precisely, the prime factorization
\[n = \prod_{j=1}^\infty p_j^{\kappa_j}\]
associates to $n$ the finite non-negative multi-index $\kappa(n)=(\kappa_1,\kappa_2,\kappa_3,\ldots)$. This means that the Bohr lift of $f$ is
\[\mathscr{B}f(z)=\sum_{n=1}^\infty a_n z^{\kappa(n)},\]
where $z = (z_1, z_2, z_3, \ldots)$.
The mapping $\mathscr{B}\colon\mathscr{H}^2\to H^2(\mathbb{T}^\infty)$ is an isometric isomorphism that respects multiplication. $\mathbb{T}^\infty$ is a compact abelian group, and its Haar measure is denoted by $m_\infty$. The measure $m_\infty$ is equal to the product of the normalized Lebesgue measure on $\mathbb{T}$ in each variable. In particular, $H^2(\mathbb{T}^d)$ is a natural subspace of $H^2(\mathbb{T}^\infty)$. We refer to \cite{HLS97,QQ13} for further properties of $H^2(\mathbb{T}^\infty)$. 

In \cite{Bayart02}, Bayart introduced the spaces $\mathscr{H}^p$, for $1\leq p<\infty$, as those Dirichlet series $f$ such that $\mathscr{B}f\in H^p(\mathbb{T}^\infty)$, and we define the $\mathscr{H}^p$-norm as
\[\|f\|_{\mathscr{H}^p} = \left(\int_{\mathbb{T}^\infty}|\mathscr{B}f(z)|^p\,dm_\infty(z)\right)^\frac{1}{p}.\]
As above, $H^p(\mathbb{T}^d)$ is a natural subspace of $H^p(\mathbb{T}^\infty)\simeq \mathscr{H}^p$.

Returning to the multiplicative Hankel form $H_\varphi$ defined in \eqref{eq:Hphidef}, the fact that $\mathscr{B}$ respects multiplication implies that
\[H_\varphi(fg) = \langle \mathscr{B}f\mathscr{B}g,\,\mathscr{B}\varphi\rangle_{H^2(\mathbb{T}^\infty)}.\]
From this representation, it is clear that we may replace $\mathscr{B}\varphi$ with any $\psi \in L^2(\mathbb{T}^\infty)$ such that $P\psi = \mathscr{B}\varphi$, where $P$ denotes the Riesz projection from $L^2(\mathbb{T}^\infty)$ to $H^2(\mathbb{T}^\infty)$. In this case, we also denote the Hankel form $H_\varphi$ by $H_\psi$. If $\psi \in L^\infty(\mathbb{T}^\infty)$, then $\|H_\varphi\|\leq\|\psi\|_\infty$, where $\|H_\varphi\|$ denotes the norm of $H_\varphi$ acting on $\mathscr{H}^2 \times \mathscr{H}^2$, and we say that $H_\varphi$ has bounded symbol $\psi$. Note that if the functional 
\begin{equation*}
\label{eq:h1bdd}
f \mapsto \langle f, \varphi \rangle, \qquad f \in \mathscr{H}^1,
\end{equation*}
 is bounded on $\mathscr{H}^1 \simeq H^1(\mathbb{T}^\infty) \subset L^1(\mathbb{T}^\infty)$, then $H_\varphi$ has a bounded symbol by the Hahn-Banach theorem. Hence, $H_\varphi$ has a bounded symbol if and only if $\varphi \in (\mathscr{H}^1)^\ast$.

The main result of \cite{OCS12} implies that there exist bounded multiplicative Hankel forms that do not have a bounded symbol. It should be pointed out that the proof is non-constructive, and no example of a bounded multiplicative Hankel form without a bounded symbol has been identified. On the other hand, if $d=1$ then Nehari's theorem \cite{Nehari} states that every bounded Hankel form $H_\varphi$ on $H^2(\mathbb{T}^d)\odot H^2(\mathbb{T}^d)$ has a bounded symbol $\psi \in L^\infty(\mathbb{T}^d)$. Nehari's theorem has been extended to $d<\infty$ by Ferguson--Lacey \cite{FL02} and Lacey--Terwilleger \cite{LT09}.

The matrix of the multiplicative Hankel form \eqref{eq:Hphidef} is
\begin{equation} \label{eq:Mrho}
	M_\varrho = \big( \varrho_{mn} \big)_{m,n\geq1} = \begin{pmatrix}
		\varrho_1 & \varrho_2 & \varrho_3 & \cdots \\
		\varrho_2 & \varrho_4 & \varrho_6 & \cdots \\
		\varrho_3 & \varrho_6 & \varrho_9 & \cdots \\
		\vdots    & \vdots    & \vdots    & \ddots
	\end{pmatrix}.
\end{equation}
By isolating the first row and column in $M_\varrho$ using the inner product representation of $H_\varphi$ from \eqref{eq:Hphidef}, we obtain
\begin{equation} \label{eq:decomp}
	H_\varphi(fg) = a_1b_1\varrho_1 + a_1\left\langle g-b_1,\varphi \right\rangle + b_1\left\langle f-a_1,\varphi\right\rangle + H_\varphi\big((f-a_1)(g-b_1)\big).
\end{equation}
The left hand side is a bounded Hankel form if and only if $\varphi \in \left(\mathscr{H}^2\odot\mathscr{H}^2\right)^\ast$, while the right hand side is bounded if and only if $\varphi \in \left(\mathscr{H}^2_0\right)^\ast = \mathscr{H}^2/\mathbb{C}$ and $\varphi \in \left(\mathscr{H}^2_0\odot\mathscr{H}^2_0\right)^\ast$. While it is obvious that 
\begin{equation} \label{eq:H23}
	\left(\mathscr{H}^2\odot\mathscr{H}^2\right)^\ast \subseteq \mathscr{H}^2,
\end{equation}
we shall now see that the corresponding statement for $\mathscr{H}^2_0$ is not true. This will follow immediately from our next result, which also is crucial in establishing the other main results of the paper.

\begin{thm}[Matrix embedding] \label{thm:matrixemb}
	Let $C=(c_{j,k})_{j,k\geq1}$ be an infinite matrix defining an operator on $\ell^2$. Consider the Dirichlet series
	\[\varphi(s) = \sum_{j=1}^\infty \sum_{k=1}^\infty  c_{j,k} \left(p_{2j-1}p_{2k}\right)^{-s},\]
	where $\{p_j\}_{j \geq 1}$ denotes the sequence of primes numbers in increasing order.
	Then
	\begin{itemize}
		\item[(a)] $\|H_\varphi\|_0 = \|C\|$,
		\item[(b)] $\|H_\varphi\| \asymp \|C\|_{S_2} = \|\varphi\|_{\mathscr{H}^2}$,
	\end{itemize}
	where $\|H_\varphi\|_0$ denotes the norm of $H_\varphi$ acting on $\mathscr{H}^2_0\times\mathscr{H}^2_0$, and $\|C\|_{S_2}$ denotes the Hilbert--Schmidt matrix norm of $C$, $$\|C\|_{S_2} = \Bigg(\sum_{j=1}^\infty \sum_{k=1}^\infty |c_{j,k}|^2\Bigg)^\frac{1}{2}.$$
\end{thm}
\begin{proof}
	Let $f,g \in \mathscr{H}^2_0$ with coefficients $\{a_j\}_{j\geq1}$ and $\{b_k\}_{k \geq 1}$, respectively. Since there are no constant terms in $\mathscr{H}^2_0$ we have that
	\begin{equation} \label{eq:Hphifg}
		H_\varphi(fg)=\left\langle fg, \varphi \right\rangle = \sum_{j=1}^\infty \sum_{k=1}^\infty\big(a_{p_{2j-1}}b_{p_{2k}}+a_{p_{2k}}b_{p_{2j-1}}\big)c_{j,k}.
	\end{equation}
	Note that for every prime $p$, $a_p$ and $b_p$ each only appear once in this sum. Let
\[\mathscr{K}_1 = \operatorname{span} \{p_{2k-1} ^{-s} \, : \, k\geq 1\}, \quad \mathscr{K}_2 = \operatorname{span} \{p_{2k} ^{-s} \, : \, k\geq 1\}, \quad \mathscr{K}_3 = \mathscr{H}^2_0 \ominus  \left(\mathscr{K}_1 \oplus \mathscr{K}_2 \right), \]
and let $P_{\mathscr{K}_j}$ denote the corresponding orthogonal projections. Let $\mathbf{a}_j$ and $\mathbf{b}_j$ denote the coefficient sequences, in the natural basis of $\mathscr{K}_j$, of $ P_{\mathscr{K}_j}f$ and $P_{\mathscr{K}_j}g$, respectively.  Then we may rewrite  \eqref{eq:Hphifg} as 
\[H_\varphi(fg) = \langle C \mathbf{b}_2, \mathbf{a}_1 \rangle_{\ell^2} +  \langle C \mathbf{a}_2, \mathbf{b}_1 \rangle_{\ell^2} = \langle \mathcal{J}(C^T \oplus C)(\mathbf{a}_1, \mathbf{a}_2), (\mathbf{b}_1, \mathbf{b}_2) \rangle_{\ell^2\oplus\ell^2}, \]
where $\mathcal{J}$ is the involution on $\ell^2\oplus\ell^2$ defined by $ \mathcal{J}(\mathbf{a}_1, \mathbf{a}_2) = (\mathbf{a}_2, \mathbf{a}_1)$. We conclude that
	\[H_\varphi\big|_{\mathscr{H}^2_0} \simeq \mathcal{J}(C^T\oplus C) \oplus 0,\] 
	completing the proof of (a). For (b), we first observe that setting $g=1$ implies $\|H_\varphi\|\geq\|\varphi\|_{(\mathscr{H}^2)^\ast}=\|\varphi\|_{\mathscr{H}^2} = \|C\|_{S_2}$. Returning to the decomposition \eqref{eq:decomp} we see that $\|H_\varphi\|\leq 4\|C\|_{S_2}$, by using (a).
\end{proof}
As a corollary of Theorem~\ref{thm:matrixemb}, we obtain that a bounded Hankel form on $\mathscr{H}^2_0 \times \mathscr{H}^2_0$ does not necessarily have a symbol in $L^2(\mathbb{T}^\infty)$, in stark contrast with the classical situation where bounded Hankel forms have bounded symbols. We also find that \eqref{eq:H23} does not hold for $\mathscr{H}^2_0$.
\begin{cor} \label{cor:bigdual}
	$\left(\mathscr{H}^2_0\odot\mathscr{H}^2_0\right)^\ast \not\subseteq \mathscr{H}^2$. That is, there are bounded multiplicative Hankel forms $H_\varphi$ on $\mathscr{H}^2_0 \times \mathscr{H}^2_0$ with the property that there is no $\psi \in L^2(\mathbb{T}^\infty)$ such that $H_\varphi = H_\psi$. 
\end{cor}
\begin{proof}
	Use Theorem~\ref{thm:matrixemb} and let $C$ be the matrix of the identity operator on $\ell^2$.
\end{proof}
Actually, we have the following stronger version of Corollary \ref{cor:bigdual}, which can be proven by considering all diagonal operators $C$ on $\ell^2$ and using Theorem~\ref{thm:matrixemb}. It exemplifies concretely that $\mathscr{H}^2_0\odot\mathscr{H}^2_0$ is in some ways significantly smaller than $\mathscr{H}^2\odot\mathscr{H}^2$.
\begin{cor} \label{cor:ell1}
The Dirichlet series
\[f(s) = \sum_{k=1}^\infty a_{k}(p_{2k-1}p_{2k})^{-s}\]
is in $\mathscr{H}^2_0\odot\mathscr{H}^2_0$ if and only if $a \in \ell^1$, while it is in $\mathscr{H}^2\odot\mathscr{H}^2$ if and only if $a \in \ell^2$.
\end{cor}
Recall that $H^2(\mathbb{T}^d)$ is a natural subspace of $H^2(\mathbb{T}^\infty)$ and that if $f \in \mathscr{H}^2_0$, then $\mathscr{B}f(0)=0$. We now observe that the inclusions behave as expected for the corresponding finite-dimensional subspaces of the weak product spaces. 
\begin{lem} \label{lem:polyincl}
	For $1 \leq d < \infty$, let $H^2_0(\mathbb{T}^d)$ denote the space of functions $F \in H^2(\mathbb{T}^d)$ for which $F(0,0,\ldots,0)=0$. Then 
	\[\big(H^2_0(\mathbb{T}^d)\odot H^2_0(\mathbb{T}^d)\big)^\ast \subseteq H^2_0(\mathbb{T}^d)\ominus\Lin(\mathbb{T}^d) \subseteq H^2_0(\mathbb{T}^d),\]
	where $\Lin(\mathbb{T}^d)$ denotes the subspace of $H_0^2(\mathbb{T}^d)$ consisting of linear functions, 
	\[\Lin(\mathbb{T}^d) = \Big\{L(z) = \sum_{j=1}^d a_j z_j \,:\, a_j \in \mathbb{C}\Big\}.\]
	\begin{proof}
		It is sufficient to show that
		\[H^2_0(\mathbb{T}^d)\ominus\Lin(\mathbb{T}^d) \subseteq H^2_0(\mathbb{T}^d)\odot H^2_0(\mathbb{T}^d),\]
		since it follows that any functional in $\big(H^2_0(\mathbb{T}^d)\odot H^2_0(\mathbb{T}^d)\big)^\ast$ must be represented by a unique element of $H^2_0(\mathbb{T}^d)\ominus\Lin(\mathbb{T}^d)$. Every $F \in H^2_0(\mathbb{T}^d)\ominus\Lin(\mathbb{T}^d)$ can be written 
		\[F(z) = \sum_{j=1}^d z_j F_j(z),\]
		where $F_j \in H^2_0(\mathbb{T}^d)$. This representation of $F$ is not unique, but we can always organize it so that $\sum_{j} \|F_j\|_{H^2(\mathbb{T}^d)}^2 = \|F\|_{H^2(\mathbb{T}^d)}^2$. By the computation 
		\[\|F\|_{H_0^2\odot H^2_0} \leq \sum_{j=1}^d 1\cdot \|F_j\|_{H^2(\mathbb{T}^d)} \leq \sqrt{d}\Big(\sum_{j=1}^d \|F_j\|_{H^2(\mathbb{T}^d)}^2\Big)^\frac{1}{2} = \sqrt{d}\|F\|_{H^2(\mathbb{T}^d)},\]
		we see that $F \in H^2_0 \odot H^2_0$.
	\end{proof}
\end{lem}

It is clear that the final part of this argument breaks down for $d=\infty$; the key point being that the subspace $\Lin(\mathbb{T}^\infty)$ of linear functions in $H^2_0(\mathbb{T}^\infty) \simeq \mathscr{H}^2_0$ is infinite-dimensional, which from the Dirichlet series point of view corresponds to the fact that there are infinitely many prime numbers. Even so, Corollary \ref{cor:bigdual} is surprising. We stress that its conclusion is related to the additional arithmetical obstructions which appear when computing the norm of an element in  $\mathscr{H}^2_0 \odot\mathscr{H}^2_0$ rather than in  $\mathscr{H}^2 \odot\mathscr{H}^2$. The following result is intended to clarify this statement. In particular, it demonstrates that the subspace of linear functions actually is complemented in $\mathscr{H}^2\odot\mathscr{H}^2$. 

\begin{thm} \label{thm:mproj}
For a non-negative integer $m$, let $P_m$ denote the projection on $m$-homogeneous Dirichlet series,
\[P_m\sum_{n=1}^\infty a_n n^{-s} = \sum_{\Omega(n)=m} a_n n^{-s}.\]
Then $P_m$ is a contraction on $\mathscr{H}^2\odot\mathscr{H}^2$.
\end{thm}
\begin{proof}
	The case $m=0$ is trivial. Let $m\geq1$ and suppose that 
	\begin{equation} \label{eq:Fk}
	F = \sum_{k} f_kg_k
	\end{equation}
	is a finite sum. Then
	\[P_mF(s) = \sum_{k} \sum_{j=0}^m P_jf_k(s)P_{m-j}g_k(s).\] 
	By applying the definition of the norm of $\mathscr{H}^2\odot\mathscr{H}^2$ and the Cauchy--Schwarz inequality, we find that
	\begin{align*}
		\|P_mF\|_{\mathscr{H}^2\odot\mathscr{H}^2} &\leq \sum_{k} \sum_{j=0}^m \|P_jf_k\|_{\mathscr{H}^2}\|P_{m-j}g_k\|_{\mathscr{H}^2} \\
		&\leq \sum_{k}\Big(\sum_{j=0}^m \|P_j f_k\|_{\mathscr{H}^2}^2 \Big)^\frac{1}{2}\Big(\sum_{j=0}^m \|P_{m-j}g_k\|_{\mathscr{H}^2}^2\Big)^\frac{1}{2} \\
		&\leq \sum_{k}\|f_k\|_{\mathscr{H}^2}\|g_k\|_{\mathscr{H}^2},
	\end{align*}
	the final inequality following from the fact that
	$$\sum_{j=0}^\infty \|P_jf\|_{\mathscr{H}^2}^2 = \|f\|_{\mathscr{H}^2}^2, \qquad f \in \mathscr{H}^2.$$
	The proof is completed by taking the infimum over the representations \eqref{eq:Fk}.
\end{proof}

We return to the matrix of $H_\varphi$ acting on $\mathscr{H}^2\times\mathscr{H}^2$ from \eqref{eq:Mrho}. The matrix $M_\rho^0$ corresponding to the action of $H_\varphi$ on $\mathscr{H}^2_0\times\mathscr{H}^2_0$ is obtained from $M_\rho$ by deleting the first row and column. That is, $M^0_\rho = (\rho_{mn})_{m,n\geq 2}$ in view of \eqref{eq:Mrho}.

Now, suppose that $H_\varphi$ is a compact form, i.e. that its matrix $M$ defines a compact operator on $\ell^2$. Let  
\[\Lambda = \{\lambda_1,\,\lambda_2,\,\ldots\}\]
denote the singular value sequence of $M$. We say that $H_\varphi$ is in the Schatten class $S_p$, $0<p\leq\infty$, if $\Lambda \in \ell^p$, and we let $\|H_\varphi\|_{S_p} = \|\Lambda\|_{\ell^p}$. When speaking of a Hankel form $H_\varphi$ we will write $S_p(\mathscr{H}^2)$ or $S_p(\mathscr{H}^2_0)$ to clarify which space is being considered; using Theorem~\ref{thm:matrixemb} as in Corollary~\ref{cor:ell1}, it is easy to construct Hankel forms belonging to the latter Schatten class, but not to the former.

Helson \cite{Helson06} showed that if $H_\varphi \in S_p(\mathscr{H}^2)$ and $p=2$, then $H_\varphi$ has a bounded symbol. In \cite{BP2014}, the authors showed that this is no longer the case when 
\[p>p_0 \approx 5.738817179.\]
We will now investigate symbols for forms $H_\varphi \in S_p(\mathscr{H}^2_0)$. We start by verifying that Helson's result still holds for $S_2(\mathscr{H}^2_0)$.

As in Lemma \ref{lem:polyincl}, any bounded Hankel form on $\mathscr{H}^2_0\times\mathscr{H}^2_0$ has a symbol $\varphi$ in $(\mathscr{H}_0^2\odot\mathscr{H}_0^2)^\ast$ of the form
\[\varphi(s) = \sum_{\Omega(n)\geq2} \varrho_n n^{-s}.\]
From this fact, a computation shows that
\begin{align*}
	\|H_\varphi\|_{S_2(\mathscr{H}_0^2)}^2 &= \sum_{m=2}^\infty \sum_{n=2}^\infty |H_\varphi (m^{-s}n^{-s})|^2 \\
	&=  \sum_{\Omega(n)\geq2} \left(d(n)-2\right)|\varrho_n|^2 \asymp \sum_{\Omega(n)\geq2} d(n) |\varrho_n|^2.
\end{align*}
Here $d(n)$ denotes the number of divisors of $n$, and the final estimate follows from the fact that $d(n)-2 \geq d(n)/3$ for $n$ such that $\Omega(n) \geq 2$, seeing as $d(n)\geq\Omega(n)+1$. Hence we can use Helson's inequality 
\[\Bigg(\sum_{n=1}^\infty \frac{|a_n|^2}{d(n)}\Bigg)^\frac{1}{2}\leq \Bigg\|\sum_{n=1}^\infty a_n n^{-s}\Bigg\|_{\mathscr{H}^1}\]
to conclude that $\varphi\in(\mathscr{H}^1)^\ast$ whenever $H_\varphi \in S_2(\mathscr{H}^2_0)$. That is, $H_\varphi$ has a bounded symbol whenever $H_\varphi \in S_2(\mathscr{H}^2_0)$.  We now show that Helson's result is optimal for $S_p(\mathscr{H}^2_0)$.

\begin{thm} \label{thm:helsonopt}
	For $p > 2$ there exist Hankel forms $H_\varphi \in S_p(\mathscr{H}^2_0)$ such that no $\psi$ in $L^2(\mathbb{T}^\infty)$ satisfies $H_\varphi = H_\psi$. In particular, there exist Hankel forms $H_\varphi \in S_p(\mathscr{H}^2_0)$ for which there are no bounded symbols.
\end{thm}

\begin{proof}
	Let $C = \left(c_{j,k} \right)_{j,k \geq 1}$ be a matrix defining an operator on $\ell^2$ which belongs to $S_p$ but not to $S_2$. In accordance with Theorem~\ref{thm:matrixemb} let
	\begin{equation*} \label{eq:Imatrix}
		\varphi(s) = \sum_{j=1}^\infty \sum_{k=1}^\infty c_{j,k} (p_{2j-1}p_{2k})^{-s}.
	\end{equation*}
	Since, as in the proof of Theorem~\ref{thm:matrixemb}, $H_{\varphi}\big|_{\mathscr{H}^2_0} \simeq \mathcal{J}\left(C^T \oplus C\right) \oplus 0$, we have that 
	\[\|H_{\varphi}\|_{S_p(\mathscr{H}_0^2)}^p = 2\|C\|_{S_p}^p < \infty.\] 
	On the other hand, we have by assumption that 
	\[\|\varphi\|_{\mathscr{H}^2} = \|C\|_{S_2} = \infty. \qedhere\]
\end{proof}
While Theorem~\ref{thm:helsonopt} does not concern Hankel forms on $\mathscr{H}^2 \times \mathscr{H}^2$, we do consider it to give us an indication that $p=2$ might be the critical value also in this case.

\begin{conj} \label{conj:schatten}
	For every $p>2$ there exists a multiplicative Hankel form $H_\varphi$ in $S_p(\mathscr{H}^2)$ without a bounded symbol.
\end{conj}

\section{A square function characterization of $\mathscr{H}^p$ and skew products} \label{sec:skew}
In the context of the classical Hardy spaces, it was Bourgain \cite{Bourgain86} who recalled the square function characterization of $H^p$ due to Fefferman and Stein \cite{FS72} and used it to the effect of showing that $H^2 \odot \partial H^2 \subseteq \partial H^1$, where $\partial H^p$ denotes the space consisting of the derivatives of all $H^p$-functions. In view of the fact that $\partial H^1 = \partial (H^2 \odot H^2) \subseteq H^2 \odot \partial H^2$ this immediately implies that 
\begin{equation}
	\label{eq:classicalproduct} 
	\partial^{-1}( H^2 \odot 
	\partial H^2) = H^2 \odot H^2. 
\end{equation}
In terms of bilinear forms, we can naturally associate a Hankel-type form $J_g$ to every element $g \in \left( 
\partial^{-1}( H^2 \odot 
\partial H^2)\right)^\ast$. If an additive Hankel form $H_g$ on $H^2 \times H^2$ corresponds to the matrix $\left(\hat{g}(j+k)\right)_{j,k\geq0}$, then $J_g$ has matrix 
\[\left( \frac{j+1}{j+k+1} \hat{g}(j+k) \right)_{j,k\geq0}.\] 
Hence Bourgain's lemma \eqref{eq:classicalproduct} can be equivalently rephrased to say that the map $H_g \mapsto J_g$ is bounded in operator norm. This statement actually carries greater interest than what its face value might suggest. The matrix 
\[\left(\frac{j+1}{j+k+1}\right)_{j,k\geq0}\]
is not a bounded Schur multiplier on all matrices, and hence the map $H_g \mapsto J_g$ is not completely bounded \cite{DP97}. This observation is at the heart of Pisier's \cite{Pisier97} construction of a polynomially bounded operator not similar to a contraction.

We define the skew product space $\partial^{-1}(\mathscr{H}^2 \odot \partial \mathscr{H}^2)$ as the Banach space completion of the space of Dirichlet series $F$ whose derivatives have a finite sum representation $F' = \sum_{k}f_k g_k'$, where $f_k, g_k \in \mathscr{H}^2$. The completion is taken under the norm
\[ \|F\|_{\partial^{-1}(\mathscr{H}^2 \odot \partial \mathscr{H}^2)} = |F(+\infty)| + \inf \sum_k \|f_k\|_{\mathscr{H}^2}\|g_k\|_{\mathscr{H}^2}, \]
where the infimum is computed over all finite representations. From the product rule $(fg)' = f'g + fg'$ it is clear that 
\begin{equation}
	\label{eq:skewproductcontainment} \mathscr{H}^2 \odot \mathscr{H}^2 \subseteq \partial^{-1}(\mathscr{H}^2 \odot \partial \mathscr{H}^2). 
\end{equation}

Our first goal is to establish a square function characterization of $\mathscr{H}^p$, for $0<p<\infty$, and use it to show that $\partial^{-1}(\mathscr{H}^2\odot \partial \mathscr{H}^2)\subseteq\mathscr{H}^1$. We begin by recalling that the spaces $\mathscr{H}^p$ are related to the Möbius invariant Hardy spaces in the right half-plane, $\mathbb{C}_0$, defined as
\[H^p_{\operatorname{i}}(\mathbb{C}_0) = \Bigg\{f \in \Hol(\mathbb{C}_0)\,:\, \|f\|_{H^p_{\operatorname{i}}(\mathbb{C}_0)}=\sup_{\sigma>0}\Big(\frac{1}{\pi}\int_{\mathbb{R}}|f(\sigma+it)|^p\frac{dt}{1+t^2}\Big)^\frac{1}{p}<\infty\Bigg\}.\]
Given a character $\chi \in \mathbb{T}^\infty$, we ``twist'' the Dirichlet series $f(s)=\sum_{n\geq1}a_n n^{-s}$ to obtain
\[f_\chi(s) = \sum_{n=1}^\infty a_n \chi(n)n^{-s}, \qquad \chi(n) = \chi^{\kappa(n)}.\]
We will require the following basic result, which can be extracted from Lemma~5 and Theorem~5 in \cite{Bayart02}.
\begin{lem} \label{lem:Hpi}
	Let $0<p<\infty$, and suppose that $f \in \mathscr{H}^p$. For almost every $\chi\in\mathbb{T}^\infty$, $f_\chi \in H^p_{\operatorname{i}}(\mathbb{C}_0)$. Moreover, 
	\[\|f\|_{\mathscr{H}^p} = \left(\int_{\mathbb{T}^\infty}\|f_\chi\|_{H^p_{\operatorname{i}}(\mathbb{C}_0)}^p\,dm_\infty(\chi)\right)^\frac{1}{p}.\]
\end{lem}
\begin{rem}
	The results in \cite{Bayart02} are stated only for $p\geq1$, but the same arguments lead to our statement of  Lemma~\ref{lem:Hpi}.
\end{rem}
For $\tau \in \mathbb{R}$, let $\Gamma_\tau$ be the cone 
\begin{equation*}
	\Gamma_\tau = \{ \sigma + i t \, : \, |t - \tau | < \sigma\} 
\end{equation*}
in the right half-plane $\mathbb{C}_0$, with vertex at $i\tau$. For a holomorphic function $f$ in $\mathbb{C}_0$, let $Sf$ be the square function, or the Lusin area integral, 
\begin{equation*}
	Sf(\tau) = \left(\int_{\Gamma_\tau} |f'(\sigma + it)|^2 \, d\sigma \, dt \right)^{1/2}, \qquad \tau \in \mathbb{R}, 
\end{equation*}
and let $f^\ast$ denote the non-tangential maximal function
\begin{equation*}
	f^\ast(\tau) = \sup_{s \in \Gamma_\tau} |f(s)|, \qquad \tau \in \mathbb{R}.
\end{equation*}
Since $1/(1+\tau^2)$ is a Muckenhoupt $A_q$-weight for all $q > 1$, it follows from Gundy and Wheeden \cite{GW7374} that $f \in H^p_{\operatorname{i}}(\mathbb{C}_0)$ if and only if $f^\ast \in L^p_{\operatorname{i}}(\mathbb{R}) = L^p \left((1+\tau^2)^{-1} \, d\tau\right)$, for $0 < p < \infty$, with comparable norms. Furthermore, if $\lim_{\sigma \to \infty} f(\sigma + it) = 0$, then 
\begin{equation}
	\label{eq:sqfunctionclassical} \|f^\ast\|_{L^p_{\operatorname{i}}(\mathbb{R})} \asymp \| Sf \|_{L^p_{\operatorname{i}}(\mathbb{R})}. 
\end{equation}
This gives us a norm expression for functions in $\mathscr{H}^p$ in terms of the square function.
\begin{thm}
	\label{thm:sqfcn} Let $f(s) = \sum_{n\geq1} a_n n^{-s}$. Then for any $0 < p < \infty$, we have 
	\begin{align}
		\label{eq:sqfcn} 
		\begin{split}
			\|f\|_{\mathscr{H}^p}^p &\asymp |a_1|^p + \int_{\mathbb{T}^\infty} \|S(f_\chi)\|^p_{L^p_{\operatorname{i}}(\mathbb{R})} \,dm_\infty(\chi)\\
			&= |a_1|^p + \int_{\mathbb{T}^\infty} \int_{\mathbb{R}}\left(\int_{\Gamma_\tau} |f_\chi'( \sigma + it)|^2 \, d\sigma \, dt \right)^{p/2} \, \frac{d\tau}{1+\tau^2} \,dm_\infty(\chi). 
		\end{split}
	\end{align}
\end{thm}
\begin{proof}
	In view of \eqref{eq:sqfunctionclassical} and Lemma~\ref{lem:Hpi} we obtain \eqref{eq:sqfcn} for $f$ with constant term $a_1 = 0$, that is, for $f \in \mathscr{H}^p_0$. Note that the linear functional $f \mapsto a_1$ is bounded on $\mathscr{H}^p$, corresponding to the functional $\mathscr{B}f \mapsto \mathscr{B}f(0)$ on $H^p(\mathbb{T}^\infty)$ \cite{CG86}. Hence, the closed subspace $\mathscr{H}_0^p$ is complemented in $\mathscr{H}^p$ by $\mathbb{C}$, and $\eqref{eq:sqfcn}$ follows in general for $f \in \mathscr{H}^p$, with one side being finite if and only if the other is.
\end{proof}
\begin{cor}
	\label{cor:skewprod} $\partial^{-1}(\mathscr{H}^2 \odot \partial \mathscr{H}^2) \subseteq \mathscr{H}^1$. 
\end{cor}
\begin{proof}
	Suppose that $f, g \in \mathscr{H}^2$, and that $F$ is the Dirichlet series such that $F' = fg'$ with $F(+\infty) = 0$. Since $\|g - g(+\infty)\|_{\mathscr{H}^2} \leq \|g\|_{\mathscr{H}^2}$ it is for the purpose of proving the statement justified to assume that $g(+\infty) = 0$. We then have that
	\begin{align*}
		\|F\|_{\mathscr{H}^1} &\asymp \int_{\mathbb{T}^\infty} \int_{\mathbb{R}}\left(\int_{\Gamma_\tau} |f_\chi( \sigma + it)|^2 |g'_\chi(\sigma +it)|^2 \, d\sigma \, dt \right)^{1/2} \, \frac{d\tau}{1+\tau^2} \,dm_\infty(\chi) \\
		&\leq \int_{\mathbb{T}^\infty} \int_{\mathbb{R}} (f_\chi)^\ast(\tau) \left(\int_{\Gamma_\tau} |g'_\chi(\sigma +it)|^2 \, d\sigma \, dt \right)^{1/2} \, \frac{d\tau}{1+\tau^2} \,dm_\infty(\chi) \\
		&\leq \int_{\mathbb{T}^\infty} \|(f_\chi)^\ast\|_{L^2_{\operatorname{i}}(\mathbb{R})} \left( \int_{\mathbb{R}} \int_{\Gamma_\tau} |g'_\chi(\sigma +it)|^2 \, d\sigma \, dt \, \frac{d\tau}{1+\tau^2} \right)^{1/2} \,dm_\infty(\chi) \\
		&\asymp \int_{\mathbb{T}^\infty} \|f_\chi\|_{H^2_{\operatorname{i}}(\mathbb{C}_0)} \|g_\chi\|_{H^2_{\operatorname{i}}(\mathbb{C}_0)} \,dm_\infty(\chi) \leq \|f\|_{\mathscr{H}^2}\|g\|_{\mathscr{H}^2}. 
	\end{align*}
	This proves that $\partial^{-1}(\mathscr{H}^2 \odot \partial \mathscr{H}^2) \subseteq \mathscr{H}^1$. 
\end{proof}

Before proceeding, we give a few remarks on the application of Theorem \ref{thm:sqfcn} to the Hardy space $H^p(\mathbb{T}^d)$ of a finite-dimensional polydisc, $d < \infty$. Let $D$ denote the differentiation operator on Dirichlet series,
\[Df(s) = \partial f(s) = f'(s) = -\sum_{n=2}^\infty a_n \log(n) n^{-s}. \]
Consider a series $f$ such that $\mathscr{B}f \in H^2(\mathbb{T}^d)$, i.e. such that $a_n = 0$ if $p_j | n$ for some $j > d$. Identifying $p_j$ with the $j$th complex variable $z_j$, the differentiation operator $D$ in the usual polydisc notation has the form
\begin{equation} \label{eq:Ddef}
	D\mathscr{B}f(z_1, \ldots, z_d) = -\sum_{j=1}^d \log(p_j) z_j \partial_{z_j} \mathscr{B}f(z_1, \ldots, z_d).
\end{equation}
Hence Theorem \ref{thm:sqfcn} gives us a new type of square function characterization of $H^p(\mathbb{T}^d)$, in terms of the differentiation operator $D$. In analogy with Corollary \ref{cor:skewprod} it can be used to prove that 
\[D^{-1}\left(H^2(\mathbb{T}^d) \odot D H^2(\mathbb{T}^d)\right) \subseteq H^1(\mathbb{T}^d)\]
and by the characterization of $H^1(\mathbb{T}^d)$ due to Ferguson--Lacey \cite{FL02} and Lacey--Terwilleger \cite{LT09} we conclude that in the finite polydisc we have
\begin{equation} \label{eq:nonradialdiff}
D^{-1}\left(H^2(\mathbb{T}^d) \odot D H^2(\mathbb{T}^d)\right) = H^2(\mathbb{T}^d) \odot H^2(\mathbb{T}^d) = H^1(\mathbb{T}^d). 
\end{equation}
It should be objected, however, that the weighted differentiation operator $D$ might not be natural in the setting of the polydisc. In Section \ref{sec:radial} we shall consider the constructs of the present section for the infinite polydisc, using the radial differentiation operator instead of $D$, and in the process prove that \eqref{eq:nonradialdiff} is valid also for radial differentiation and integration.

We return to the discussion of products of Dirichlet series spaces, and note that Corollary~\ref{cor:skewprod} in combination with \eqref{eq:skewproductcontainment} yields that
\begin{equation} \label{eq:productincl2}
\mathscr{H}^2\odot\mathscr{H}^2 \subseteq \partial^{-1}\left(\partial\mathscr{H}^2\odot\mathscr{H}^2\right) \subseteq \mathscr{H}^1.
\end{equation}
The remainder of this section is devoted to the investigation of whether these inclusions are strict. We begin with the following observation.
\begin{lem}
	\label{lem:pisiercond} Let $\varphi(s) = \sum_{k\geq1} \overline{\rho_k} k^{-s}$ be a function in $\mathscr{H}^2$. Then $\varphi$ induces a bounded linear functional $\upsilon_\varphi$ on $\partial^{-1}(\mathscr{H}^2 \odot \partial \mathscr{H}^2)$, via the $\mathscr{H}^2$-pairing, if and only if the form 
	\begin{equation}
		\label{eq:skewform} J_\varphi(a,b) = \sum_{m=1}^\infty \sum_{n=1}^\infty a_m b_n \frac{\log n}{\log m + \log n} \rho_{mn} 
	\end{equation}
	is bounded on $\ell^2 \times \ell^2$, where the summand is understood to be $0$ if $m=n=1$. The corresponding norms are equivalent,
	\[\|\upsilon_\varphi\| \asymp |\rho_1| + \|J_\varphi\|.\]
	In particular, if $\rho_k \geq 0$ for all $k$, then $\varphi \in \left( 
	\partial^{-1}(\mathscr{H}^2 \odot 
	\partial \mathscr{H}^2) \right)^*$ if and only if $\varphi \in \left( \mathscr{H}^2 \odot \mathscr{H}^2 \right)^*$, with equivalent norms. 
\end{lem}
\begin{proof}
	Suppose that $f$ and $g$ are Dirichlet series with coefficient sequences $a$ and $b$, respectively. Let $\partial^{-1} (f'g)$ denote the primitive of $f'g$ with constant term $0$. Then
	\[\langle \partial^{-1} (f'g), \varphi \rangle = \sum_{m=1}^\infty \sum_{n=1}^\infty a_m b_n \frac{\log n}{\log m + \log n} \rho_{mn}, \]
	proving the first part of the proposition. For the second part, note as per usual that the action of $\varphi$ as an element in $\left( \mathscr{H}^2 \odot \mathscr{H}^2 \right)^*$ corresponds to the multiplicative Hankel form 
	\begin{equation}
		\label{eq:multhankform} H_\varphi(a,b) = \sum_{m=1}^\infty \sum_{n=1}^\infty a_m b_n \rho_{mn}. 
	\end{equation}
	Hence, if $\rho_k \geq 0$ for all $k$, then 
	\begin{equation*}
		\|\varphi\|_{\left( 
		\partial^{-1}(\mathscr{H}^2 \odot 
		\partial \mathscr{H}^2) \right)^*} \ll \|\varphi\|_{ \left( \mathscr{H}^2 \odot \mathscr{H}^2 \right)^*}. 
	\end{equation*}
	The converse inequality is a direct consequence of \eqref{eq:skewproductcontainment}. 
\end{proof}

Ortega-Cerd\'a and Seip \cite{OCS12} showed that $\mathscr{H}^2 \odot \mathscr{H}^2 \subsetneq \mathscr{H}^1$. With Lemma~\ref{lem:pisiercond}, we are able to apply their technique to prove the corresponding statement for $\partial^{-1}(\mathscr{H}^2 \odot \partial \mathscr{H}^2)$. 
\begin{thm}
	$\partial^{-1}(\mathscr{H}^2 \odot \partial \mathscr{H}^2) \subsetneq \mathscr{H}^1$ 
\end{thm}
\begin{proof}
	Let $d$ be a positive integer and consider the function
	\[ \varphi_d(s) = \prod_{j=1}^d \left(p_{2j-1}^{-s} + p_{2j}^{-s}\right),\]
	where $\{p_j\}_{j\geq 1}$ again denotes the prime sequence.
	The norm of $\varphi_d$ as an element of the dual of $\mathscr{H}^2 \odot \mathscr{H}^2$ is $2^{d/2}$ \cite{OCS12}. Since the coefficients of $\varphi_d$ are non-negative, Lemma~\ref{lem:pisiercond} hence shows that
	\[ \|\varphi_d\|_{\left(\partial^{-1}(\mathscr{H}^2 \odot \partial \mathscr{H}^2) \right)^*} \asymp 2^{d/2}.\]
	On the other hand, consider $f_d = \varphi_d$ as an element of $\mathscr{H}^1$, $\|f_d\|_{\mathscr{H}^1} = (4/\pi)^d$ \cite{OCS12}. Since $\langle f_d, \varphi_d \rangle_2 = 2^d$, the functional induced by $\varphi_d$ on $\mathscr{H}^1$ has norm at least $(\pi/2)^d$. If it were the case that $\partial^{-1}(\mathscr{H}^2 \odot \partial \mathscr{H}^2) = \mathscr{H}^1$, then the norm of $\varphi_d$ as a functional on $\mathscr{H}^1$ and the norm as a functional on $\partial^{-1}\left(\mathscr{H}^2 \odot 
	\partial \mathscr{H}^2\right)$ would be equivalent, a contradiction as $d \to \infty$.
\end{proof}

The remaining question of whether
\begin{equation}
	\label{eq:producteq} \left(
	\partial^{-1}(\mathscr{H}^2 \odot 
	\partial \mathscr{H}^2)\right)^* = \left( \mathscr{H}^2 \odot \mathscr{H}^2 \right)^* 
\end{equation}
or, equivalently, whether the first inclusion in \eqref{eq:productincl2} is strict, appears to be subtle.  As we just saw in Lemma \ref{lem:pisiercond} it can be rephrased as to ask if the forms \eqref{eq:skewform} and \eqref{eq:multhankform} are simultaneously bounded, which would mean precisely that \[ \left(\frac{\log n}{\log m + \log n} \right)_{m,n \geq 1}\]
is a Schur multiplier on the class of multiplicative Hankel forms. Specializing to the one-variable case by only considering integers of the form $2^k$, we see that the analogue of \eqref{eq:producteq} for the classical Hardy space $H^2(\mathbb{T})$ is equivalent to the statement that $(j+1)/(j+k+1)$ is a Schur multiplier on (additive) Hankel forms, as discussed in the introduction of this section.

However, by applying Theorem~\ref{thm:matrixemb} in full force together with Schur multiplier techniques, we are able to show that the inclusion is strict when $\mathscr{H}^2$ is replaced by $\mathscr{H}^2_0$. We define  $\partial^{-1}\left(\mathscr{H}^2_0\odot\partial\mathscr{H}^2_0\right)$ in exact analogy with our previous considerations, except that we impose all of its elements $f$ to have constant term $f(+\infty) = 0$.

\begin{thm} \label{thm:H20strict}
	$\mathscr{H}^2_0\odot\mathscr{H}^2_0 \subsetneq \partial^{-1}\left(\mathscr{H}^2_0\odot\partial\mathscr{H}^2_0\right)$.
\end{thm}
\begin{proof}
	Assume to the contrary that 
	\[\left(\frac{\log n}{\log m + \log n} \right)_{m,n \geq 2}\]
	is a Schur multiplier on bounded multiplicative Hankel forms 
	\[\rho(a,b) = \sum_{m=2}^\infty \sum_{n=2}^\infty a_m b_n \rho_{mn}, \qquad a,b \in \ell^2.\]
	Applied to every symbol constructed by the procedure of Theorem~\ref{thm:matrixemb}, we conclude that 
	\begin{equation} \label{eq:bennet}
		\left(\frac{\log{p_{2j-1}}}{\log{p_{2k}+\log{p_{2j-1}}}}\right)_{j,k \geq 1}
	\end{equation}
	is a Schur multiplier on all matrices $C$ defining bounded operators $C : \ell^2 \to \ell^2$. However, \eqref{eq:bennet} cannot be a Schur multiplier, as this would defy Bennett's criterion \cite{Bennett77}, since
	\[\lim_{j\to\infty} \lim_{k\to\infty}\frac{\log{p_{2j-1}}}{\log{p_{2k}+\log{p_{2j-1}}}} = 0,\]
	while
	\[\lim_{k\to\infty} \lim_{j\to\infty}\frac{\log{p_{2j-1}}}{\log{p_{2k}+\log{p_{2j-1}}}} = 1. \qedhere\]
\end{proof}
	
It must be stressed that Theorem \ref{thm:H20strict} \emph{does not} imply that the inclusion in \eqref{eq:skewproductcontainment} is strict. If we attempt to apply the proof to $\mathscr{H}^2\odot\mathscr{H}^2$, the matrices constructed by Theorem~\ref{thm:matrixemb} are Hilbert--Schmidt. To be a Schur multiplier on Hilbert--Schmidt matrices means only to have bounded entries, so no contradiction is obtained. However, we do feel that Theorem \ref{thm:H20strict} invokes the natural conjecture.

\begin{conj} \label{conj:skew}
	The inclusion between the standard weak product and its skew counterpart is strict, $\mathscr{H}^2\odot\mathscr{H}^2 \subsetneq \partial^{-1}\left(\mathscr{H}^2\odot\partial\mathscr{H}^2\right)$.
\end{conj}

\section{Radial differentiation} \label{sec:radial}
From the polydisc point of view, the constructs of the last section all arose from the weighted differentiation operator $D$ of \eqref{eq:Ddef}, obtained from the Dirichlet series formalism. In the present section we shall consider instead the more natural radial differentiation operator of equation \eqref{eq:Rdiff}. Before commencing, note that as in Theorem~\ref{thm:mproj} every Dirichlet series may be decomposed into $m$-homogeneous subseries,
\[f(s) = \sum_{n=0}^\infty a_n n^{-s} = \sum_{m=0}^\infty \Big(\sum_{\Omega(n)=m} a_n n^{-s}\Big) = \sum_{m=0}^\infty P_mf(s).\]
Through the Bohr lift, this is equivalent to the corresponding decomposition of a power series in a countably infinite number of variables,
\[F(z) = \sum_{n=0}^\infty a_n z^{\kappa(n)} = \sum_{m=0}^\infty \Big(\sum_{|\kappa(n)|=m} a_n z^{\kappa(n)}\Big) = \sum_{m=0}^\infty P_m F(z),\qquad z = (z_1,z_2,\ldots).\]
We recall that $\kappa(n) = (\kappa_1, \kappa_2, \ldots)$ is the finitely supported multi-index associated to every positive integer $n$ through its prime decomposition, so that
\[|\kappa(n)| = \Omega(n) = \sum_j \kappa_j.\]
Consider now, for any $z \in \mathbb{T}^\infty$, the following power series in one variable $w$.
\[F_z(w) = F(z w) = \sum_{n=1}^\infty a_n z^{\kappa(n)} w^{\Omega(n)} = \sum_{m=0}^\infty P_m F(z)w^m.\]
Observe in particular that the $m$th coefficient of $F_z(w)$ is the $m$-homogeneous subseries of $F$. From here it is clear that differentiation in the auxiliary variable $w$ allows us to capture the natural radial differentiation of the polydisc, since every monomial of order $m$ is treated equally. This is further justified by the formal computation
\[w\frac{d}{dw}F_z(w) = w\sum_{j=1}^\infty z_j \partial_{z_j} F(wz) = (RF)_z(w).\]

We have the following analogue of Lemma~\ref{lem:Hpi}. We also point out that through the Bohr lift a similar statement can be made for Dirichlet series.
\begin{lem} \label{lem:wnorm}
Let $F \in H^p(\mathbb{T}^\infty)$, $0 < p < \infty$. Then $F_z \in H^p(\mathbb{T})$ for almost every $z \in \mathbb{T}^\infty$ and
\begin{equation} \label{eq:wnorm}
	\|F\|_{H^p(\mathbb{T}^\infty)} = \left( \int_{\mathbb{T}^\infty} \|F_z\|_{H^p(\mathbb{T})}^p \, dm_\infty(z) \right)^{1/p}. 
\end{equation}
\end{lem}
\begin{proof}
This follows from Fubini's theorem and the fact that $z \mapsto F(z)$ and $z\mapsto F(wz)$, for $w \in \mathbb{T}$, have equal $H^p(\mathbb{T}^\infty)$-norm.
\end{proof}
For $\theta \in [0, 2\pi)$, let $\Gamma_\alpha(\theta)$ denote the Stolz angle in $\mathbb{D}$ with vertex at $e^{i\theta}$ and of some fixed aperture $\alpha < \pi/2$. The (slightly non-standard) square function $Sg$ of a function $g$ holomorphic in $\mathbb{D}$ is given by
\[Sg(\theta) = \left( \int_{\Gamma_\alpha(\theta)} |wg'(w)|^2 \, dA(w) \right)^{1/2},\]
where $dA$ denotes the normalized area element. If $g(0)=0$ we have that $\|g\|_{H^p(\mathbb{T})} \asymp \|Sg\|_{L^p(\mathbb{T})}$. Since $F_z(0)=F(0)$ for every $z \in \mathbb{T}^\infty$, this immediately gives us the analogue of Theorem \ref{thm:sqfcn}.
\begin{thm}
Let $F \in H^p(\mathbb{T}^\infty)$, $0 < p < \infty$. Then
\[\|F\|_{H^p(\mathbb{T}^\infty)}^p \asymp |F(0)|^p + \int_{\mathbb{T}^\infty} \int_0^{2\pi}\left(\int_{\Gamma_\alpha(\theta)} |(RF)_z(w)|^2 \, dA(w) \right)^{p/2} \, \frac{d\theta}{2\pi} \, dm_\infty(z) .\]
\end{thm}
Now most of the arguments of the previous section can be repeated. We collect the results that follow without providing details. Note in particular the satisfying conclusion obtained for the finite-dimensional polydisc. Indeed, this result partly motivates the existence of this section. 
\begin{cor} \label{cor:radial}
We have that 
\begin{equation*} 
H^2(\mathbb{T}^\infty) \odot H^2(\mathbb{T}^\infty) \subseteq R^{-1}\left(H^2(\mathbb{T}^\infty) \odot R H^2(\mathbb{T}^\infty)\right) \subsetneq H^1(\mathbb{T}^\infty). 
\end{equation*}
On the other hand, when $d < \infty$ it holds that
\begin{equation*} 
H^2(\mathbb{T}^d) \odot H^2(\mathbb{T}^d) = R^{-1}\left(H^2(\mathbb{T}^d) \odot R H^2(\mathbb{T}^d)\right) = H^1(\mathbb{T}^d). 
\end{equation*}
\end{cor} 
We remark that it is not clear how to obtain Corollary \ref{cor:radial} directly from the considerations in Section \ref{sec:skew}, due to the weights $\log p_j$ entering into Dirichlet series differentiation. In fact, suppose that $n = \prod_{j} p_j^{\kappa_j}.$
Then
\[\log{n} = \sum_{j} \kappa_j\log{p_j}\qquad\text{and}\qquad \Omega(n) = \sum_{j} \kappa_j,\]
illustrating the fact that the $R$ treats every prime equally, while the half-plane differentiation operator $D$ does not. In particular, the proof of Theorem~\ref{thm:H20strict} does not yield any information when $D$ is replaced by $R$, since the Schur multiplier vital to the proof has entries
\[\frac{\Omega(n)}{\Omega(m) + \Omega(n)} = \frac{\Omega(p_{2j-1})}{\Omega(p_{2k}) + \Omega(p_{2j-1})} = \frac{1}{2}.\]

It should also be pointed out that decomposing Dirichlet series (or power series on the infinite polydisc) into homogeneous subseries is not a new idea. It dates back at least to Bohnenblust--Hille \cite{BH31}, and has recently been applied to obtain results for composition operators on spaces of Dirichlet series \cite{BB14} as well as $L^1$-estimates for Dirichlet polynomials \cite{BS14}.

We conclude this paper by providing a charming inequality, which follows at once from Lemma~\ref{lem:wnorm} and the classical Hardy inequality
\begin{equation} \label{eq:hardyineq}
	\sum_{m=0}^\infty \frac{|b_m|}{m+1} \leq \pi \Bigg\|\sum_{m=0}^\infty b_m w^m\Bigg\|_{H^1(\mathbb{T})}.
\end{equation}
\begin{cor} \label{cor:hardyhomo}
	Let $f(s)=\sum_{n\geq1}a_n n^{-s} \in \mathscr{H}^1$ and consider the $m$-homogeneous subseries $P_m f(s) = \sum_{\Omega(n)=m}a_n n^{-s}$. Then
	\[\sum_{m=0}^\infty \frac{\|P_m f\|_{\mathscr{H}^1}}{m+1}\leq \pi \|f\|_{\mathscr{H}^1}.\] 
\end{cor}
Corollary~\ref{cor:hardyhomo} can be compared to the estimate $\|P_m f\|_{\mathscr{H}^1}\leq\|f\|_{\mathscr{H}^1}$ appearing in \cite[Lem.~3]{BS14}. Returning to the beginnings of this paper, we mention that Hardy's inequality \eqref{eq:hardyineq} in turn can be obtained by viewing the bounded symbol for the sharpest version of Hilbert's inequality \eqref{eq:neharisymbol} as an element in the dual of $H^1(\mathbb{T})$ (see \cite[pp.~47--49]{Duren}).

\bibliographystyle{amsplain} 
\bibliography{product} 
\end{document}